\documentclass[12pt]{amsart}
\usepackage{amsmath,amssymb,amsfonts}
\usepackage[all]{xy}

\usepackage[T5,T1]{fontenc}
\usepackage{mathpazo}
\usepackage{hyperref}
\usepackage{a4wide}

\theoremstyle{plain}
\newtheorem{thm}{Theorem}[section]
\newtheorem{lem}[thm]{Lemma}
\newtheorem{cor}[thm]{Corollary}
\newtheorem{prop}[thm]{Proposition}

\theoremstyle{definition}
\newtheorem{defn}[thm]{Definition}

\newtheorem{rmk}[thm]{Remark}

\newcommand{\im}{{\rm im}}

\newcommand{\Gal}{{\rm Gal}}

\newcommand{\sC}{{\mathcal C}}

\newcommand{\sE}{{\mathcal E}}

\newcommand{\sG}{{\mathcal G}}

\newcommand{\sM}{{\mathcal M}}

\newcommand{\C}{{\mathbb C}}

\newcommand{\F}{{\mathbb F}}

\newcommand{\Q}{{\mathbb Q}}
\newcommand{\R}{{\mathbb R}}

\newcommand{\U}{{\mathbb U}}

\newcommand{\Z}{{\mathbb Z}}

\def\NDT{{\fontencoding{T5}\selectfont Nguy\~ \ecircumflex n Duy T\^an}}
 
\begin{document}
\title{Triple Massey Products over global fields} 

 \author{ J\'an Min\'a\v{c} and \NDT}
\address{Department of Mathematics, Western University, London, Ontario, Canada N6A 5B7}
\email{minac@uwo.ca}
 \address{Department of Mathematics, Western University, London, Ontario, Canada N6A 5B7 and Institute of Mathematics, Vietnam Academy of Science and Technology, 18 Hoang Quoc Viet, 10307, Hanoi - Vietnam } 
\email{dnguy25@uwo.ca}
\thanks{JM is partially supported  by the Natural Sciences and Engineering Research Council of Canada (NSERC) grant R0370A01. NDT is partially supported  by the National Foundation for Science and Technology Development (NAFOSTED) grant 101.04-2014.34}
 \begin{abstract}
 Let $K$ be a global field which contains a primitive $p$-th root of unity, where $p$ is a prime number. M. J. Hopkins and K. G. Wickelgren showed that for $p=2$, any triple Massey product over $K$ with respect to $\F_p$, contains 0 whenever it is defined. We show that this is true for all primes $p$.
\end{abstract}
\maketitle
\section{Introduction}
Massey products were introduced by W. S. Massey in \cite{M}. (We recall the definition in  Section 2.) Massey products were first used in topology where usual cohomology cup products would not detect some linking properties of knots but Massey products would.  (See for example \cite[page 98]{Mo} or \cite[pages 154-158]{GM}.) Further interest in Massey products arises as an obstruction to the "formality" of manifolds over real numbers.  In the case of compact K\"ahler manifolds, formality formalizes the property that their homotopy type is a formal consequence of their real cohomology ring. (See \cite{DGMS}.) We treat Massey products also as obstructions to solving certain Galois embedding problems.

Throughout this paper, we let $p$ be a prime number.  Let $K$ be a field which we assume  contains a fixed primitive $p$-th root of unity $\zeta_p$. Let  $G_K$ be the absolute  Galois group of $K$. Let $\sC^\bullet =\sC^\bullet(G_K,\F_p)$ denote the differential graded algebra of $\F_p$-inhomogeneous cochains in continuous group cohomology of $G_K$ (see e.g., \cite[Chapter I, \S 2]{NSW}).  For any $a\in K^\times=K\setminus \{0\}$, let $\chi_a$ denote the corresponding character  via the Kummer map $K^\times \to H^1(G_K,\F_p)$, i.e., $\chi_a$ is defined by $ \sigma (\sqrt[p]{a})= \zeta_p^{\chi_a(\sigma)} \sqrt[p]{a},$ for all $\sigma\in G_K$. 
In the work of M. J. Hopkins and K. G. Wickelgren \cite{HW}, the following fundamental result was proved. (By a {global field} we mean a finite extension of $\Q$, or a function field in one variable over a finite field.)

\begin{thm}[{\cite[Theorem 1.2]{HW}}]
\label{thm:HW} Let the notation be as above. Assume that $p=2$ and 
 $K$ is a global field of characteristic $\not=2$. Let $a,b,c\in K^\times$. The triple Massey product $\langle \chi_a,\chi_b,\chi_c\rangle$ contains 0 whenever it is defined.
\end{thm} 
 
In \cite{MT} we extend the result of Hopkins-Wickelgren to an arbitrary field $K$ of characteristic $\not=2$, still assuming that $p=2$. 

\begin{thm}[{\cite[Theorem 1.2]{MT}}]
\label{thm:vanishing} Let the notation be as above. Assume that $p=2$ and 
 $K$ is an arbitrary field of characteristic $\not=2$. Let $a,b,c\in K^\times$.
 The triple Massey product $\langle \chi_a,\chi_b,\chi_c\rangle$ contains 0 whenever it is defined.
\end{thm}

In this paper we extend the result of  Hopkins-Wickelgren in Theorem~\ref{thm:HW} in another direction. We still consider a global field $K$ but we let the prime $p$ be arbitrary. 
\begin{thm} 
\label{thm:main}
 Let the notation be as above. Assume that $K$ is a global field containing a primitive $p$-th root of unity and $a,b,c\in K^\times$. Then the triple Massey product $\langle \chi_a,\chi_b,\chi_c\rangle$ contains 0 whenever it is defined.
\end{thm}
Let us denote by $\U_{4}(\F_p)$  the group of all upper-triangular unipotent $4$-by-$4$-matrices with entries in  $\F_p$.
 For a finite group $G$, by a $G$-Galois extension $L/K$, we mean a Galois extension with Galois group isomorphic to $G$. It is a classical problem to describe extensions $M/K$ which can be embedded into a $G$-Galois extension $L/K$ with a prescribed Galois group $G$. 
From Theorem~\ref{thm:main} and its local version we can deduce the following contribution to this problem when $G=\U_4(\F_p)$.
\begin{cor}
\label{cor:U4}
 Let $K$ be a local or global field containing a primitive $p$-th root of unity. Let $a,b,c\in K^\times$ and  assume that  the classes $[a],[b],[c]$ in the $\F_p$-vector space $K^\times/(K^\times)^p$ are linearly independent. Assume further that $\chi_a\cup\chi_b=\chi_b\cup\chi_c=0$ in $H^2(G_K,\F_p)$.  Then the Galois extension $K(\sqrt[p]{a},\sqrt[p]{b}, \sqrt[p]{c})/K$ can be embedded in a $\U_4(\F_p)$-Galois extension $L/K$.
\end{cor}
In fact for each $\U_4(\F_p)$-extension $L/K$, there exist $a,b,c\in K^\times\cap L^p$ such that  the classes $[a],[b],[c]$ in the $\F_p$-vector space $K^\times/(K^\times)^p$ are linearly independent, and that $\chi_a\cup\chi_b=\chi_b\cup\chi_c=0$ in $H^2(G_K,\F_p)$. Thus we see that this hypothesis is both necessary and sufficient for embedding abelian extensions of degree $p^3$ and exponent $p$ into a $\U_4(\F_p)$-extension. (See Section 4 for more detail.) 

In the case when $p=2$, Corollary~\ref{cor:U4} was also proved  in \cite[Section 4]{GLMS} for all fields $K$ of characteristic not 2. (See also \cite[Section 6]{MT}.)

Let us now recall briefly how Theorem~\ref{thm:HW} is established in \cite{HW}. 

Let $p=2$ and $K$ be a field of characteristic not $2$. In \cite{HW}, the authors construct for each $a,b,c\in K^\times$, a  $K$-variety $X_{a,b,c}$ which 
 has a $K$-rational point if and only if  the triple Massey product $\langle \chi_a,\chi_b,\chi_c\rangle$ is defined and contains 0 (see \cite[Theorem 1.1]{HW}).  The authors then establish a local version of Theorem~\ref{thm:HW} by using the non-degeneracy property of the cup products and the indeterminacy of Massey products.
Now assume that $K$ is a global field and consider $a,b,c\in K^\times$ such that $\langle \chi_a,\chi_b,\chi_c\rangle$ is defined. By applying a result of D. B. Leep and A. R. Wadsworth in \cite{LW}, the authors show that the splitting variety $X_{a.b,c}$ satisfies the Hasse local-global principle (see \cite[Theorem 3.4]{HW}), and then the result follows from the local case.

In our paper we also use the local-global principle but our method is different from the method used in the paper \cite{HW}.
Let $p$ be any prime, and let $K$ be a field containing a primitive $p$-th root of unity. Let $a,b,c\in K^\times$ such that the triple Massey product $\langle \chi_a,\chi_b,\chi_c\rangle$  is defined.
Now instead of constructing a splitting variety for $\langle \chi_a,\chi_b,\chi_c\rangle$, we use the technique of Galois embedding problems to  detect the vanishing property of triple Massey products.
Namely, $\langle \chi_a,\chi_b,\chi_c\rangle$ vanishes if certain kinds of embedding problems are solvable. This is true because of a result of W. G. Dwyer.  
We then use Hoechsmann's lemma to translate the problem of solvability of embedding problems to the problem of showing the vanishing of some degree 2 cohomology classes. Then we establish a local-global principle for the vanishing of the cohomology classes (see Lemma~\ref{lem:local-global}). Theorem~\ref{thm:main} then follows from its local version. This being said, our proof also provides another proof for Theorem~\ref{thm:HW} in the case $p=2$. 
\\
\\
{\bf Acknowledgments: } We would like to thank  Stefan Gille, Thong Nguyen Quang Do and Kirsten Wickelgren for their interest and correspondence. 
We are  grateful to the an anonymous referee for his/her very careful reading of our paper and for providing us with insightful comments and valuable suggestions which we used to improve our exposition considerably. For example, Proposition~\ref{prop:local-global} and Lemma~\ref{lem:1} were formulated based on his/her report.
\\
\\
{\it Addendum (October 2015)}: Since submitting of this paper there have been some new significant developments in this subject motivated and influenced by this paper and \cite{MT}. In \cite{EM1}  Efrat and Matzri proved a result which implies the main result  Theorem~\ref{thm:main} of this paper. In \cite{Ma} Matzri extended our main result Theorem~\ref{thm:main} to an arbitrary field $K$. Efrat and Matzri \cite{EM2} and in parallel \cite{MT3} gave a direct proofs of Matzri's result, using only tools from Galois cohomology.  In \cite{MT4} the explicit constructions of $\U_4(\F_p)$-Galois extensions over all fields which admit such extensions are provided. 
In \cite{MT5} the authors also considered  the vanishing property of higher Massey products over rigid fields.

\section{Review of Massey products}
In this section, we review some basic facts about Massey products,  see \cite{MT} and references therein for more detail.

Let $A$ be a unital commutative ring. Recall that a differential graded algebra (DGA) over $A$ is a graded associative $A$-algebra 
\[\sC^\bullet =\oplus _{k\geq 0} \sC^k=\sC^0\oplus\sC^1\oplus \sC^2\oplus \cdots \] with product $\cup$ and equipped with a differential $\partial\colon \sC^\bullet\to \sC^{\bullet +1}$ such that 
\begin{enumerate}
\item $\partial$ is a derivation, i.e.,
\[
\partial(a\cup b) = \partial a\cup b +(-1)^k a\cup \partial b \quad (a\in \sC^k);
\]
\item $\partial^2=0$.
\end{enumerate}
Then as usual the cohomology $H^\bullet$ of $\sC^\bullet$ is $\ker\partial/\im\partial$. We shall assume that $a_1,\ldots,a_n$ are elements in $H^1$.
\begin{defn}
 A collection $\sM=(a_{ij})$, $1\leq i<j\leq n+1$, $(i,j)\not=(1,n+1)$ of elements of $\sC^1$ is called a {\it defining system} for the {\it $n$-fold  Massey product} $\langle a_1,\ldots,a_n\rangle$ if the following conditions are fulfilled:
\begin{enumerate}
\item $a_{i,i+1}$ represents $a_i$.
\item $\partial a_{ij}= \sum_{l=i+1}^{j-1} a_{il}\cup a_{lj}$ for $i+1<j$.
\end{enumerate}
Then $\sum_{k=2}^{n} a_{1k}\cup a_{k,n+1}$ is a $2$-cocycle. (See for example \cite[page 233]{Fe}.)
Its  cohomology class in $H^2$  is called the {\it value} of the product relative to the defining system $M$,
and is denoted by $\langle a_1,\ldots,a_n\rangle_\sM$.

The product $\langle a_1,\ldots, a_n\rangle$ itself is the subset of $H^2$ consisting of all  elements which can be written in the form $\langle a_1,\ldots, a_n\rangle_\sM$ for some defining system $\sM$. 

The $n$-fold Massey product $\langle a_1,\ldots, a_n\rangle$ is said to be {\it defined} if it has a defining system, i.e., the set $\langle a_1,\ldots, a_n\rangle$ is non-empty.

For $n\geq2$ we say that $\sC^\bullet$ has the {\it vanishing $n$-fold Massey product property}
if every defined Massey product $\langle a_1,\ldots, a_n\rangle$, where $ a_1,\ldots, a_n\in \sC^1$,
necessarily contains $0$. When $n=3$ we will speak about {\it triple} Massey products and the vanishing {\it triple} Massey product property.
\end{defn}

Now let  $G$ be a profinite group and let $A$ be a finite commutative ring considered as a trivial discrete $G$-module. Let $\sC^\bullet=\sC^\bullet(G,A)$ be the DGA of inhomogeneous continuous cochains of $G$ with coefficients in $A$ \cite[Ch.\ I, \S2]{NSW}. We write $H^i(G,A)$ for the corresponding cohomology groups.

\begin{defn}
 We say that $G$ has the {\it vanishing $n$-fold Massey product property (with respect to $A$)} if the DGA $\sC^\bullet(G,A)$ has the vanishing $n$-fold Massey product property.
\end{defn}
\section{Unipotent matrices}

Let $\U_{n+1}(\F_p)$ be the group of all upper-triangular unipotent $(n+1)\times(n+1)$-matrices with entries in  $\F_p$. Let $Z_{n+1}(\F_p)$ be the subgroup of all such matrices with all off-diagonal entries being $0$ except possibly at position $(1,n+1)$. We may identify the quotient $\U_{n+1}(\F_p)/Z_{n+1}(\F_p)$ with the group $\bar\U_{n+1}(\F_p)$ of all upper-triangular unipotent $(n+1)\times(n+1)$-matrices with entries over $\F_p$ with the $(1,n+1)$-entry omitted.

For a representation $\rho\colon G\to \U_{n+1}(\F_p)$ and $1\leq i< j\leq n+1$, let $\rho_{ij}\colon G\to \F_p$ be the composition of $\rho$ with the projection from $\U_{n+1}(\F_p)$ to its $(i,j)$-coordinate. We use  similar notation for representations $\bar\rho\colon G\to\bar\U_{n+1}(\F_p)$. Note that $\rho_{i,i+1}$ (resp., $\bar\rho_{i,i+1}$) is a group homomorphism. 

Now we assume $n=3$. We consider the following exact sequence of finite groups
\[
1 \longrightarrow A \longrightarrow \U_4(\F_p) \xrightarrow{(a_{12},a_{23},a_{34})} \F_p^3\longrightarrow 1,
\]
here $a_{ij}\colon \U_4(\F_p)\to\F_p$ is the map sending a matrix to its $(i,j)$-coefficient. Explicitly,
\[
A=\left\{ \begin{bmatrix} 
1& 0 & a & b\\
0& 1 & 0 & c\\
0 & 0 & 1& 0\\
0& 0 & 0 & 1
\end{bmatrix} : a,b,c\in \F_p\right\}.
\] 

We consider the action of $\U_4(\F_p)$ on $A$ by conjugation: $g\cdot a= g a g^{-1}$, $\forall \, g\in \U_4(\F_p), a\in A$. Since $A$ is abelian, this action induces  an action of $\F_p^3$ on $A$, i.e., we get a homomorphism $\psi\colon \F_p^3\to {\rm Aut}(A)$.

Let $A^\prime={\rm Hom}(A,\F_p)$ be the dual $\F_p^3$-module of the $\F_p^3$-module $A$. Here the action of $\F_p^3$ on $A^\prime$ is given by
\[
(g\phi)(a)= \phi(g^{-1}\cdot a),
\] 
where $\phi\in {\rm Hom}(A,\F_p)$, $g\in \F_p^3$ and $a\in A$. (Here we write the group $\F_p^3$ multiplicatively.) 
From this action, we get a homomorphism $\psi^\prime \colon \F_p^3\to {\rm Aut}(A^\prime)$. 

The following lemma is a special case of a more general result on matrix representations of dual representations. For the convenience of the reader, we include a short proof.

\begin{lem}
\label{lem:dual rep}
Assume that $\{e_1,e_2,e_3\}$ is a basis for the $\F_p$-vector space $A$. 
Let $g$ be any element $\F_p^3$. Suppose that  $\psi(g)$ is given by matrix $X$ with respect to $e_1,e_2,e_3$. Then the matrix of $\psi^\prime(g)$ with respect to the dual basis is $(X^{-1})^T$.
\end{lem}
\begin{proof} We write $X^{-1}=(x_{ij})$. Let $\{e_1^\prime,e_2^\prime,e_3^\prime\}$ be the dual basis of the basis $\{e_1,e_2,e_3\}$. Then 
\[
(\psi^{\prime}(g)(e^\prime_i))(e_j)=e^\prime_i(\psi(g^{-1})(e_j))=e^\prime_i(\sum_k x_{kj} e_k)= x_{ij}=(\sum_k x_{ik} e_k^\prime)(e_j).
\]
Hence $\psi^{\prime}(g)(e^\prime_i)=\sum_k x_{ik} e_k^\prime$, and the lemma follows.
\end{proof}
\begin{lem}
\label{lem:explicit}
There exists an $\F_p$-basis of $A^\prime$ such that with respect to this basis the map $\psi^\prime  \colon \F_p^3\to {\rm Aut}(A^\prime)$ becomes a  map $ \F_p^3\to {\rm GL_3}(\F_p)$  which sends $(x,y,z)\in \F_p^3$ to 
$\begin{bmatrix}
1 & 0 &-x\\
0& 1 & z\\
 0& 0 &1
\end{bmatrix}.$
\end{lem}
\begin{proof}
We first describe the action of $\F_p^3$ on $A$, i.e., we describe the map $\psi\colon \F_p^3\to {\rm Aut}(A)$, as follows. 

Let $e_1=I+ E_{24}$, $e_2=I+E_{13}$,  $e_3=I+ E_{14}$. 
We have
\[
\psi(x,y,z)(e_1)= 
\begin{bmatrix}
1 & x & 0 &0 \\
0& 1 & y & 0\\
0& 0 & 1 & z \\
0& 0 & 0 & 1
\end{bmatrix}
\begin{bmatrix}
1 & 0 & 0 &0 \\
0& 1 & 0 & 1\\
0& 0 & 1 & 0 \\
0& 0 & 0 & 1
\end{bmatrix}
\begin{bmatrix}
1 & x & 0 &0 \\
0& 1 & y & 0\\
0& 0 & 1 & z \\
0& 0 & 0 & 1
\end{bmatrix}^{-1}
= \begin{bmatrix}
1 & 0 & 0 &x \\
0& 1 & 0 & 1\\
0& 0 & 1 & 0 \\
0& 0 & 0 & 1
\end{bmatrix}=e_1+xe_3;
\]
\[
\psi(x,y,z)(e_2)= 
\begin{bmatrix}
1 & x & 0 &0 \\
0& 1 & y & 0\\
0& 0 & 1 & z \\
0& 0 & 0 & 1
\end{bmatrix}
\begin{bmatrix}
1 & 0 & 1 &0 \\
0& 1 & 0 & 0\\
0& 0 & 1 & 0 \\
0& 0 & 0 & 1
\end{bmatrix}
\begin{bmatrix}
1 & x & 0 &0 \\
0& 1 & y & 0\\
0& 0 & 1 & z \\
0& 0 & 0 & 1
\end{bmatrix}^{-1}
= \begin{bmatrix}
1 & 0 & 1 &-z \\
0& 1 & 0 & 0\\
0& 0 & 1 & 0 \\
0& 0 & 0 & 1
\end{bmatrix}=e_2-ze_3;
\]
\[
\psi(x,y,z)(e_3)= 
\begin{bmatrix}
1 & x & 0 &0 \\
0& 1 & y & 0\\
0& 0 & 1 & z \\
0& 0 & 0 & 1
\end{bmatrix}
\begin{bmatrix}
1 & 0 & 0 &1 \\
0& 1 & 0 & 0\\
0& 0 & 1 & 0 \\
0& 0 & 0 & 1
\end{bmatrix}
\begin{bmatrix}
1 & x & 0 &0 \\
0& 1 & y & 0\\
0& 0 & 1 & z \\
0& 0 & 0 & 1
\end{bmatrix}^{-1}
= \begin{bmatrix}
1 & 0 & 0 &1 \\
0& 1 & 0 & 0\\
0& 0 & 1 & 0 \\
0& 0 & 0 & 1
\end{bmatrix}=e_3.
\]
Thus with respect to the $\F_p$-basis $\{e_1,e_2,e_3\}$ of $A$, the element $(x,y,z)\in \F_p^3$ is sent to the matrix 
$
\begin{bmatrix}
1 & 0 & 0\\
0& 1 & 0\\
x & -z &1
\end{bmatrix} \in {\rm GL}_3(\F_p).
$

Now we consider the $\F_p^3$-module $A^\prime$. By Lemma~\ref{lem:dual rep},  the structure map $\psi^\prime \colon \F_p^3\to {\rm Aut}(A^\prime)$ describing the action of $\F_p^3$ on $A^\prime$ with respect to the dual basis of $(e_1, e_2,e_3)$, is given by: 
\[ (x,y,z)
\mapsto 
\left( \begin{bmatrix}
1 & 0 &0\\
0& 1 & 0\\
 x& -z &1
\end{bmatrix}^{-1}\right)^T= \begin{bmatrix}
1 & 0 &x\\
0& 1 & -z\\
 0& 0 &1
\end{bmatrix}^{-1}=\begin{bmatrix}
1 & 0 &-x\\
0& 1 & z\\
 0& 0 &1
\end{bmatrix}.
\] 
\end{proof}

\section{Embedding problems}
A {\it weak embedding problem} $\sE$ for a profinite group $G$ is a diagram 
\[
\sE:=
\xymatrix
{
{} & G \ar[d]^{\alpha}\\
U \ar[r]^f & \bar U
}
\] 
which consists of  profinite groups $U$ and $\bar U$ and homomorphisms $\alpha \colon G\to \bar U$, $f\colon U\to \bar U$ with $f$ being surjective. (All homomorphisms of profinite groups considered in this paper are assumed to be continuous.) 

A {\it weak solution} of $\sE$ is  a homomorphism $\beta\colon G\to U$ such that $f\beta=\alpha$. 

 We call $\sE$ a {\it finite} weak embedding problem if  the group $U$ is finite. The {\it kernel} of $\sE$ is defined to be $M:=\ker(f)$. 

Let $\phi_1\colon G_1\to G$ be a homomorphism of profinite groups. Then $\phi_1$ induces the following weak embedding problem 
\[
\sE_1:=
\xymatrix
{
{} & G_1 \ar[d]^{\alpha\circ\phi_1}\\
U \ar[r]^{f} & \bar U.
}
\]
If  $\beta$ is a weak solution of $\sE$ then $\beta\circ \phi_1$ is a weak solution of $\sE_1$.

The following result is due to W. Dwyer. We will use this result to reformulate the vanishing Massey product property in terms of weak embedding problems. 
\begin{thm}[{\cite[Theorem 2.4]{Dwy}}]
\label{thm:Dwyer}
Let $\alpha_1,\ldots,\alpha_n$ be elements of $H^1(G,\F_p)$. There is a one-one correspondence $\sM\leftrightarrow \bar\rho_{\sM}$ between defining systems ${\sM}$ for $\langle \alpha_1,\ldots,\alpha_n\rangle$ and group homomorphisms $\bar\rho_{\sM} \colon G\to \bar{\U}_{n+1}(\F_p)$ with $(\bar\rho_{\sM})_{i,i+1}= -\alpha_i$, for $1\leq i\leq n$.

Moreover $\langle \alpha_1,\ldots,\alpha_n\rangle_{\sM}=0$ in $H^2(G,\F_p)$ if and only if the dotted homomorphism exists in the following  commutative diagram
\[
\xymatrix{
& & &G \ar@{->}[d]^{\bar\rho_{\sM}} \ar@{-->}[ld]\\
0\ar[r]&\F_p\ar[r] &\U_{n+1}(\F_p)\ar[r] &\bar{\U}_{n+1}(\F_p)\ar[r] &1.
}
\]
\end{thm}
Explicitly, the one-one correspondence in Theorem~\ref{thm:Dwyer} is given by: For a defining system $\sM=(a_{ij})$ for $\langle \alpha_1,\ldots,\alpha_n\rangle$, $\bar\rho_{\sM}\colon G\to \bar{\U}_{n+1}(\F_p)$ is defined by letting $(\bar\rho_{\sM})_{ij}=-a_{ij}$ (see \cite[Proof of Theorem 2.4]{Dwy}).

\begin{lem}
\label{lem:vanishingEP}
Let $G$ be a profinite group, and $n\geq 3$ an integer. Then the following statements are equivalent:
\begin{enumerate}
\item $G$ has the vanishing $n$-fold Massey product property with respect to $\F_p$.
\item For every homomorphism $\bar\rho\colon G\to \bar\U_{n+1}(\F_p)$, the finite weak embedding problem 
\[
\xymatrix{
& & &G \ar@{->}[d]^{(\bar\rho_{12},\ldots,\bar\rho_{n,n+1})} \ar@{-->}[ld]\\
0\ar[r]& A \ar[r] &\U_{n+1}(\F_p)\ar[r] &\F_p^n\ar[r] &1,
}
\]
has a weak  solution, i.e., $(\bar\rho_{12},\bar\rho_{23},\ldots,\bar\rho_{n,n+1})$ can be lifted to a homomorphism $\rho\colon G\to \U_{n+1}(\F_p)$.
\end{enumerate}
\end{lem}

\begin{proof} This follows from Theorem~\ref{thm:Dwyer}.
\end{proof}

\begin{cor}
\label{cor:surj to U4} Let $G$ be a profinite group.  Let $\chi_1,\chi_2,\chi_3\in H^1(G,\F_p)$ be $\F_p$-linearly independent. Assume that $G$ has the vanishing triple Massey product and that $\chi_1\cup\chi_2=\chi_2\cup\chi_3=0\in H^2(G,\F_p)$. Then there is a continuous surjective homomorphism $\rho\colon G\to \U_4(\F_p)$ such that
$\rho_{12}=\chi_1$, $\rho_{23}=\chi_2$ and $\rho_{34}=\chi_3$.
\end{cor}
\begin{proof}
Since $\chi_1\cup\chi_2=\chi_2\cup\chi_3=0\in H^2(G,\F_p)$, there exist $a_{12}, a_{23}\in \sC^1(G,\F_p)$ such that $\partial a_{12}=\chi_1\cup \chi_2$ and $\partial a_{23}=\chi_2\cup\chi_3$. This implies that the triple Massey product $\langle \chi_1,\chi_2,\chi_3\rangle$ is defined. By Theorem~\ref{thm:Dwyer}, we have a homomorphism $\bar\rho\colon G\to \bar\U_4(\F_p)$ such that $\bar{\rho}_{12}=\chi_1$, $\bar{\rho}_{23}=\chi_2$ and $\bar{\rho}_{34}=\chi_3$.
By Lemma~\ref{lem:vanishingEP}, there exists a homomorphism $\rho\colon G\to \U_4(\F_p)$ such that
\[
\rho_{12}=\bar{\rho}_{12}=\chi_1,\;\; \rho_{23}=\bar{\rho}_{23}=\chi_2,\;\; \rho_{34}=\bar{\rho}_{34}=\chi_3.
\]
Note that the Frattini subgroup of $\U_4(\F_p)$ is $A$. Hence by the Frattini argument $\rho \colon G\to \U_4(\F_p)$ is surjective. 
\end{proof}
\begin{rmk}
Let $\rho\colon G\to \U_4(\F_p)$ be a surjective homomorphism. Let $\chi_1=\rho_{12}$, $\chi_2=\rho_{23}$ and $\chi_3=\rho_{34}$. Since $(\rho_{12},\rho_{23},\rho_{34})\colon G\to \F_p\times\F_p\times \F_p$ is surjective, we see that $\chi_1,\chi_2$ and $\chi_3$ are $\F_p$-linearly independent. Furthermore since $\rho$ is group homomorphism, we see that $\chi_1\cup\chi_2=\chi_2\cup\chi_3=0\in H^2(G,\F_p)$.
\end{rmk}
\begin{lem}[Hoechsmann]
\label{lem:Hoechsmann} Let $\sE$ be a finite weak embedding problem for $G$ with finite abelian kernel $M$. 
Let $\epsilon\in H^2(\bar U,M)$ be the cohomology class corresponding to the embedding problem  $\sE$. Then $\sE$ has a weak solution if and only if $\alpha^*(\epsilon)=0\in H^2(G,M)$.
\end{lem}
\begin{proof} 
See \cite[Statement 1.1, page 82]{Ho}. (See also  \cite[Chapter 3, \S 5, Proposition 3.5.9]{NSW}.)
\end{proof}
\begin{cor}
\label{cor:local-global} Let $\sE(G)=(\alpha\colon G\to \bar{U}, f\colon U\to \bar U)$ be a finite weak embedding problem for $G$ with abelian kernel $M$. Let $\phi_i\colon G_i\to G, i\in I,$ be a family of homomorphisms of profinite groups. Assume that the natural homomorphism
\[ H^2(G,M) \to \prod_i H^2(G_i,M),
\]
is injective. Then the weak embedding problem $\sE(G)$ has a weak solution if and only if for every $i\in I$ the induced weak embedding problem $\sE(G_i)$ has a weak solution.
\end{cor}
\begin{proof} 
We consider the following sequence
\[
\xymatrix{
H^2(\bar U,M) \ar@{->}[r]^{\alpha^*}&  H^2(G,M) \;\ar@{>->}[r] 
& \prod_{i\in I}  H^2(G_i, M).
}
\]
The statement follows from Lemma~\ref{lem:Hoechsmann}.
\end{proof}

\begin{prop}
\label{prop:local-global}
Suppose that $G_i$, $i\in I$, are closed subgroups of a profinite group $G$, and that for every map $\alpha\colon G\to \F_p^3$ the map 
\[
{\rm Res}\colon H^2(G,A) \longrightarrow  \prod_{i\in I}H^2(G_i,A)
\]
is injective,  where the action is via $ \psi\circ \alpha \colon G\to {\rm Aut}(A)$. If each $G_i$ has the triple vanishing Massey product property, then $G$ also has the triple vanishing Massey product property. 
\end{prop}

\begin{proof}
We shall prove the condition (2) in Lemma \ref{lem:vanishingEP}. 

Let $\bar\rho\colon G\to \bar\U_4(\F_p)$ be any homomorphism. We   consider the weak embedding problem
\[
(\sE) \xymatrix{
& & &G \ar@{->}[d]^{(\bar\rho_{12}, \bar\rho_{23}, \bar\rho_{34})} \\
0\ar[r]& A \ar[r] &\U_4(\F_p)\ar[r] &(\F_p)^3\ar[r] &1.
}
\]

By assumption  for every $i\in I$ the induced weak embedding problem $(\sE_i)$ 
\[
(\sE_i) \xymatrix{
& & &G_{i} \ar@{->}[d]^{(\bar\rho_{12}, \bar\rho_{23}, \bar\rho_{34})} \ar@{-->}[ld]\\
0\ar[r]& A \ar[r] &\U_4(\F_p)\ar[r] &(\F_p)^3\ar[r] &1,
}
\]
 has a weak solution.  By Corollary~\ref{cor:local-global}, $(\sE)$ has a weak solution also. 
\end{proof}

\section{The vanishing of a certain cohomology group}
Let $G$ be a profinite group, and let $M$ be a discrete $G$-module. We define
\[
H^1_*(G,M)=\ker ( H^1(G,M)\to \prod_{C} H^1(C,M)),
\]
where the product is over all closed cyclic subgroups (in the profinite sense) of $G$. 

(The definition of $H^1_*(G,M)$ is due to Tate (see \cite[\S 2]{Se}). This definition also appeared in \cite[\S 2]{DZ}, in which the authors used the notation $H^1_{\rm loc}$ instead of using $H^1_*$.)

The following lemma is a special case of  \cite[Lemma 3.3]{DZ}. It is a simple lemma and therefore we also omit a proof.
\begin{lem}
\label{lem:linear algebra}
 Let $V$ be a vector space of finite dimension over a field $k$. Let $\varphi_1,\varphi_2$ be elements in the dual $k$-vector space $V^*:={\rm Hom}(V,k)$. If $\ker \varphi_1\subseteq \ker\varphi_2$ then there exists $\lambda\in k$ such that $\varphi_2=\lambda \varphi_1$.
\end{lem}
\begin{lem} 
\label{lem:calculation}
Let 
\[\sG=\left\{ \begin{bmatrix} 
1 & 0 & a\\
0 & 1 & b\\
0& 0 & 1
\end{bmatrix}:
a,b\in \F_p
\right\}, 
\]
and let $\F_p^3$ act on $\sG$ by matrix multiplication.
Then $H^1_*(\sG,(\F_p)^3)=0$.
\end{lem}
\begin{proof}
Let $(Z_\sigma)$ be a cocycle representing an element in $H^1_*(\sG,(\F_p)^3)$. Then for each $\sigma \in \sG$, there exists $W_\sigma\in (\F_p)^3$ such that
\[ Z_\sigma= (\sigma-1)W_\sigma.
\]
Writing $Z_\sigma=\begin{bmatrix} x_\sigma\\ y_\sigma\\z_\sigma \end{bmatrix}$, $W_\sigma=\begin{bmatrix} u_\sigma\\ v_\sigma \\ t_\sigma \end{bmatrix}$
 and 
$\sigma= \begin{bmatrix}
1 & 0 & a_\sigma\\
0 & 1 & b_\sigma\\
0 & 0 & 1
\end{bmatrix},$ we have
\[
\begin{bmatrix} x_\sigma\\ y_\sigma\\z_\sigma \end{bmatrix} =
 \begin{bmatrix}
0 & 0 & a_\sigma\\
0 & 0 & b_\sigma\\
0 & 0 & 0
\end{bmatrix}
\begin{bmatrix} u_\sigma\\ v_\sigma \\ t_\sigma \end{bmatrix}
= \begin{bmatrix} t_\sigma a_\sigma \\ t_\sigma b_\sigma \\ 0 \end{bmatrix}.
\] 
Hence 
\[
\label{eq:1}
\tag{1}
x_{\sigma}=t_\sigma a_\sigma, y_\sigma= t_\sigma b_\sigma, z_\sigma=0.
\]
 By the cocycle condition, $\sigma\mapsto x_\sigma$ and $\sigma\mapsto y_\sigma$ are homomorphisms. Also, $\sigma\mapsto a_\sigma$ and $\sigma\mapsto b_{\sigma}$ are homomorphisms.
From (\ref{eq:1}), one has $\ker a_\sigma\subseteq \ker x_\sigma$ and $\ker b_\sigma\subseteq \ker y_\sigma$. 
Hence by Lemma~\ref{lem:linear algebra}, there exist $\lambda,\mu\in \F_p$ such that 
\[
\label{eq:2}
\tag{2}
x_\sigma= \lambda a_\sigma; y_\sigma= \mu b_\sigma.
\]
We consider the matrix $\sigma_0= \begin{bmatrix}
1 & 0 & 1\\
0 & 1 & 1\\
0 & 0 & 1
\end{bmatrix},$ i.e., $a_{\sigma_0}=b_{\sigma_0}=1$. Then (\ref{eq:1}) and (\ref{eq:2}) imply that
\[
x_{\sigma_0}=t_{\sigma_0}=\lambda, \text{ and } y_{\sigma_0}=t_{\sigma_0}=\mu.
\]
Thus $\lambda=\mu$. Hence  for all $\sigma \in \sG$ we have $Z_\sigma=(\sigma-1) W$, with $W=(0,0,\lambda)^t$. Therefore $(Z_\sigma)$ is cohomologous to 0, as desired.
\end{proof}

\section{The injectivity of a localization map}
Let $K$ be a global field containing a primitive $p$-th root of unity.
 For any $G_K$-module $M$ with the structure map $\rho\colon G_K\to {\rm Aut}(M)$, let $K(M)$ be the smallest splitting field of $M$, explicitly $K(M)$ is the fixed field of the separable  closure $K^{\rm sep}$ under $\ker(\rho)$. 
 For each prime $v$ of $K$, let $K_v$ denote the completion of $K$ at $v$. We will fix an embedding $\iota_v\colon G_{K_v}\hookrightarrow G_K$ which is induced by choosing an embedding of $K^{\rm sep}$ in $K_v^{\rm sep}$. Then for each $i$, $\iota_v$'s induce a homomorphism 
\[
\beta^1(K,M)\colon H^i(G_K,M)\to \prod_{v}H^i(G_{K_v},M).
\] 
This homomorphism does not depend on the choice of  embeddings  $K^{\rm sep}\hookrightarrow K_v^{\rm sep}$, and it is called the {\it localization map}.

\begin{lem}
\label{lem:1}
 Let $F$ be a finite Galois extension of $K$ containing $K(M)$. Then we can inject the group $\ker \beta^1(K,M)$ into the group $H^1_*({\rm Gal}(F/K),M)$.
\end{lem}
(See \cite[Proposition 8]{Se} for a similar statement.)
\begin{proof}

By \cite[Chapter I, Lemma 9.3]{Mi} and/or \cite[Lemma 1]{Ja}, we have the following diagram
\[
\xymatrix
{
&\ker\beta^1(K,M) \ar@{^{(}->}[d]\\
H^1_*(\Gal(F/K),M)\ar@{->}[r]^\simeq &H^1_*(G_K,M).
}
\]
The lemma then follows.
\end{proof}

Now let $\alpha \colon G_K\to \F_p^3$ be any (continuous) homomorphism. We consider $A$ as a $G_K$-module via 
\[
\psi\circ \alpha \colon G_K\stackrel{\alpha}{\to} \F_p^3\stackrel{\psi}{\to} {\rm Aut}(A).\]

\begin{lem}
\label{lem:local-global}
The localization map 
\[ 
H^2(G_K,A)\to \prod_{v} H^2(G_{K_v},A),
 \]
 is injective.
\end{lem}

\begin{proof}
First note that if we consider $A^\prime={\rm Hom}(A,\F_p)$ as a $G_K$-module via the composition map $ \beta=\psi^\prime \circ \alpha \colon G_K\to \F_p^3\stackrel{\psi^\prime}{\to} {\rm Aut}(A^\prime)$, then $A^\prime$ is the dual $G_K$-module of the $G_K$-module $A$. We shall choose an $\F_p$-basis of  $A^\prime$ as in Lemma~\ref{lem:calculation}. Clearly, after identifying $A^\prime$ with $\F_p^3$, and ${\rm Aut}(A^\prime)$ with ${\rm GL}_3(\F_p)$, the action of $G_K$ on $A^\prime$ via the image $\im(\beta)$ is the matrix multiplication.

By Poitou-Tate duality (\cite[Theorem 8.6.7]{NSW}), it is enough to show that 
\[
\label{eq:3}
\tag{3}
\ker(H^1(G_K,A^\prime)\to \prod_{v} H^1(G_{K_v},A^\prime))=0.
\] 

Let $F= (K^{\rm sep})^{\ker\beta}$ be the smallest splitting field of $A^\prime$. 
Then $\Gal(F/K)\simeq\im (\beta)\subseteq \im\psi^\prime=\sG$, where $\sG$ is the group defined in Lemma~\ref{lem:calculation}. Here the equality $\im\psi^\prime=\sG$ follows from Lemma~\ref{lem:explicit}.

If $\Gal(F/K)\simeq \im\beta=\sG$, then by Lemma~\ref{lem:calculation},  $H^1_*(\Gal(F/K),A^\prime)=0$.
If $\Gal(F/K)\simeq \im \beta \not= \sG$, then $\Gal(F/K)$ is of order dividing $p$ because $|\sG|=p^2$. Thus ${\rm Gal}(F/K)$ is cyclic. In this case, it is clear that $H^1_*(\Gal(F/K),A^\prime)=0$.
Thus in all cases we have $H^1_*(\Gal(F/K),A^\prime)=0$. Therefore  Lemma~\ref{lem:1} implies that (\ref{eq:3}) is true, as desired. 
\end{proof}


\section{Triple Massey products over local and global fields}

Recall that  a pro-$p$-group $G$ is call a Demushkin group if its cohomology $H^i(G,\F_p)$ has the following properties: (1) $\dim_{\F_p}H^1(G,\F_p)<\infty$, (2) $\dim_{\F_p}H^2(G,\F_p)=1$ and (3) the cup product $H^1(G,\F_p)\times H^1(G,\F_p)\to H^2(G,\F_p)$ is non-degenerate.
\begin{thm}
\label{thm:local}
 Let $K$ be a local field containing a primitive $p$-th root of unity. Let $n$ be an integer greater than 2. Then every  $n$-fold Massey product  contains 0 whenever it is defined.
\end{thm}
\begin{proof}
Let $G=G_K(p)$ be the maximal pro-$p$ quotient of the absolute Galois group of $K$. If either $K\simeq \C$  or $K\simeq \R$ and  $p\not=2$, then $G$ is trivial. Clearly then $G$ has the vanishing $n$-fold Massey product property. 

If $K\simeq \R$ and $p=2$ then $G\simeq \Z/2\Z$, which is a Demushkin group by   \cite[Proposition 3.9.10]{NSW}. 
Now assume that $K$ is not isomorphic to either $\R$ or $\C$, then by \cite[Proposition 7.5.9 and Theorem 7.5.11]{NSW}, $G$ is also a Demushkin group.   
 Hence,  in the both main cases when $G$ is non-trivial,  $G$ has the vanishing$n$-fold Massey product property by \cite[Theorem 4.3]{MT}.
\end{proof}

\begin{proof}[Proof of Theorem~\ref{thm:main}]
Theorem~\ref{thm:main} follows from  Proposition~\ref{prop:local-global}, Lemma~\ref{lem:local-global} and Theorem~\ref{thm:local}.
\end{proof}
\begin{proof}[Proof of Corollary~\ref{cor:U4}]
Corollary~\ref{cor:U4} follows from Theorems~\ref{thm:local}-\ref{thm:main} and Corollary~\ref{cor:surj to U4}.
\end{proof}

\end{document}